\documentclass{amsart}

\usepackage{amsmath,amsthm, amssymb,epic,graphicx,mathrsfs,enumerate,listings}

\newtheorem{lemma}{Lemma}
\newtheorem{teor}{Theorem}

\newtheorem{cor}{Corollary}

\DeclareMathOperator{\Maxdim}{Maxdim}
\DeclareMathOperator{\Mindim}{Mindim}

\DeclareMathOperator{\PSL}{PSL}
\DeclareMathOperator{\PGL}{PGL}
\DeclareMathOperator{\soc}{soc}
\DeclareMathOperator{\ASL}{ASL}
\DeclareMathOperator{\AGL}{AGL}
\newcommand\PGammaL{{\rm P}\Gamma{\rm L}}

\begin{document}

\title{Maximal irredundant families of minimal size in the alternating group}
\date{}
\author{Martino Garonzi}
\address[Martino Garonzi]{Departamento de Matem\'atica, Universidade de Bras\'ilia, Campus Universit\'ario Darcy Ribeiro, Bras\'ilia-DF, 70910-900, Brazil}
\email{mgaronzi@gmail.com}
\thanks{MG acknowledges the support of the Funda\c{c}\~{a}o de Apoio \`a Pesquisa do Distrito Federal (FAPDF) and the Coordena\c{c}\~{a}o de Aperfei\c{c}oamento de Pessoal de N\'ivel Superior - Brasil (CAPES)}
\author{Andrea Lucchini}
\address[Andrea Lucchini]{Dipartimento di Matematica \lq\lq Tullio Levi-Civita", Universit\`a di Padova, Via Trieste 63, 35131 Padova, Italy}
\email{lucchini@math.unipd.it}
 
\begin{abstract}
Let $G$ be a finite group. A family $\mathcal{M}$ of maximal subgroups of $G$ is called ``irredundant'' if its intersection is not equal to the intersection of any proper subfamily. $\mathcal{M}$ is called ``maximal irredundant'' if $\mathcal{M}$ is irredundant and it is not properly contained in any other irredundant family. We denote by $\Mindim(G)$ the minimal size of a maximal irredundant family of $G$. In this paper we compute $\Mindim(G)$ when $G$ is the alternating group on $n$ letters.
\end{abstract}

\maketitle

For $G$ a finite group and $\mathcal{M}$ a family of maximal subgroups of $G$, $\mathcal{M}$ is called irredundant if its intersection is not equal to the intersection of any proper subfamily of $\mathcal{M}$. The ``maximal dimension'' of $G$, denoted $\Maxdim(G)$, is the maximal size of an irredundant family of $G$. R. Fernando in \cite{fernando} observed that setting $m(G)$ to be the maximal size of a generating set of $G$ that does not properly contain any other generating set and $i(G) := \max \{m(H)\ :\ H \leq G\}$, we have $$m(G) \leq \Maxdim(G) \leq i(G).$$ As Fernando observed, the work of Whiston \cite{whiston} easily implies that we have the following equalities: $m(S_n)=i(S_n)=n-1$, $m(A_n)=i(A_n)=n-2$ (in Fernando's terminology, $S_n$ and $A_n$ are flat). An immediate consequence is that $\Maxdim(S_n)=n-1$ and $\Maxdim(A_n)=n-2$.

\ 

Here we introduce the dual concept of ``minimal dimension'': $\Mindim(G)$ is the minimal size of a maximal irredundant family of $G$. When $G$ is a primitive permutation group, that is, a group acting faithfully and transitively on a set $\Omega$ and whose point stabilizers are maximal subgroups, an easy way to bound $\Mindim(G)$ from above is given by the base size of $G$. A base of $G$ is a subset $B$ of $\Omega$ such that the intersection of the stabilizers in $G$ of the elements of $B$ is trivial. The minimal cardinality of a base of $G$ is denoted $b(G)$. Since any redundant family contains an irredundant subfamily with the same intersection, and any irredundant family intersecting in $\{1\}$ is a maximal irredundant family, we obtain that $\Mindim(G) \leq b(G)$. By looking at the existent literature (considering base sizes of non-standard actions \cite{conjcam}) it is easy to prove that $\Mindim(S) \leq 7$ for every finite nonabelian simple group $S$. Indeed in the case of groups of Lie type we can always build up a non-standard action by taking primitive actions with stabilizer in the Aschbacher class $\mathscr{C}_3$, stabilizers of field extensions of the form $GL_m(q^r)$ in $GL_n(q)$, where $r$ is a prime dividing $n$ and $m=n/r$ (see \cite{kl} Tables 3.5.H and 4.3.A). 

\ 

In this paper we compute $\Mindim(A_n)$ where $A_n$ is the alternating group on $n$ letters. Our main result is the following.

\begin{teor} \label{main}
Let $n \geq 4$ be an integer. Then $\Mindim(A_n) \in \{2,3\}$ and we have $\Mindim(A_n)=3$ if and only if one of the following happens:
\begin{itemize}
\item $n=2p$ with $p \neq 11$ an odd prime and $n$ is not of the form $q+1$ where $q$ is a prime power.
\item $n \in \{6,7,8,11,12\}$.
\end{itemize}
\end{teor}

Explicitly determining for which odd primes $p$ we have $\Mindim(A_{2p})=3$ is a hard number-theoretical problem: one needs to solve the equation $2p=q+1$ where $p$ is a prime and $q$ is a prime power.

\section{Preliminary lemmas}

Before proving the main result we need some lemmas.

\begin{lemma} \label{arit2p}
Let $p$ be an odd prime and let $n=2p$. Then the following are equivalent.
\begin{enumerate}
\item $S_n$ has primitive subgroups different from $A_n$ and $S_n$.
\item $A_n$ has primitive subgroups different from $A_n$.
\item $S_n$ has primitive maximal subgroups different from $A_n$.
\item $A_n$ has primitive maximal subgroups.
\item $n$ is equal to $22$ or of the form $q+1$ where $q$ is a prime power.
\end{enumerate}
\end{lemma}

\begin{proof}
$(2)$ and $(3)$ clearly imply $(1)$, and $(2) \Leftrightarrow (4)$ is clear as a permutation group containing a primitive one is primitive. We prove that $(5)$ implies $(1)$, $(2)$, $(3)$ and $(4)$. Suppose $(5)$ holds. If $n=22$ then $A_n$ contains a primitive Mathieu subgroup $\mbox{M}_{22}$ and this gives the result, now assume $n \neq 22$. Consider $G=\PGL_2(q) \leq S_n$ where $n=2p=q+1$, $q$ a prime power, and $G$ is acting in the usual way on the projective line. Since $p \geq 3$ we have $q \geq 5$ hence $G$ is an almost-simple group with socle $\PSL_2(q)$. Being $2$-transitive, $G$ is a primitive subgroup of $S_n$. We now prove that $G \cap A_n = \PSL_2(q)$. The inclusion $\supseteq$ is clear as $G \cap A_n$ is a normal subgroup of $G$ of index at most $2$. We are left to prove that $G$ is not contained in $A_n$. Let $g \in \PGL_2(q)$ be the element corresponding to the diagonal matrix having $u$ and $1$ on the diagonal, where $u$ is a generator of the multiplicative group $\mathbb{F}_q^{\ast}$. $g$ fixes exactly two $1$-dimensional subspaces and cyclically permutes the remaining $q-1=2p-2$, now since $2p-2$ is even $g$ is an odd permutation, therefore $G$ is a primitive subgroup of $S_n$ not contained in $A_n$.

\ 

We are left to prove that $(1) \Rightarrow (5)$. Let $G$ be a primitive group of degree $n=2p$. It was shown by Wielandt \cite{wie} that $G$ has rank $2$ or $3$, and later using the classification of finite simple groups it was shown that $G$ is $2$-transitive (that is, has rank $2$) with the only exception of $p=5$ (see \cite{cl2p}). We can use the known classification of $2$-transitive groups. Suppose $p \neq 5$. Since $G$ is not affine (because $n$ is not a prime power) $G$ belongs to \cite[Table 7.4]{cam} and it follows that $n$ is equal to $22$ or is of one of the following types.
\begin{itemize}
\item $(q^d-1)/(q-1)$, $d \geq 2$, $(d,q) \neq (2,2),(2,3)$, $q$ a prime power. We have $2p=(q^d-1)/(q-1)$, hence $d$ is a prime because if $d=ab$ with $a,b$ integers then $2p=(q^{ab}-1)/(q^b-1) \cdot (q^b-1)/(q-1)$, a contradiction. On the other hand $(q^d-1)/(q-1) = 1+q+q^2+\ldots+q^{d-1}$ is odd if $d$ is odd, therefore $d=2$ and $2p=q+1$.
\item $q^2+1$, $q=2^{2d+1}>2$. This cannot happen being $2p$ even and $q^2+1$ odd.
\item $q^3+1$, $q \geq 3$ a prime power. Then $2p=q^3+1=(q+1)(q^2-q+1)$ shows that this cannot happen.
\end{itemize}
If $p=5$ then $n=10$ and $S_n$ has the primitive maximal subgroup $\PGammaL_2(9)$, so the result holds in this case as $10=9+1$.
\end{proof}

Observe that from the above proof it follows that

\begin{cor}
Let $p$ be an odd prime and let $G \leq S_{2p}$ be a primitive group of degree $2p$ not containing the alternating group $A_{2p}$. Then one of the following holds.
\begin{itemize}
\item $p=11$ and $\soc(G) \cong \mbox{M}_{22}$ in its primitive action of degree $22$.
\item $2p-1=q \geq 5$ is a prime power and $\soc(G)$ is isomorphic to $\PSL_2(q)$ acting on the projective line.
\end{itemize}
\end{cor}

\section{Proof of Theorem \ref{main}}

\begin{lemma} \label{small}
$\Mindim(A_n) \leq 3$ for all $n \geq 4$ even and for $n=7$, $n=11$. Moreover $\Mindim(A_n)=3$ for $n \in \{6,7,8,11,12\}$.
\end{lemma}

\begin{proof}
Suppose $n=2m \geq 4$ is even. Consider the dihedral group $D$ of order $n$ embedded (via the regular representation) in $S_n$. This gives two subgroups, $H$ and $K$ (the left and right regular representations), of $S_n$ which commute pointwise. Since $H$ is transitive, $C_{S_n}(H)$ is semiregular, in particular its order is at most $n$, implying that $C_{S_n}(H)=K$ and similarly $C_{S_n}(K)=H$. In particular $H$ is an intersection of centralizers of two fixed-point-free involutions, which are maximal subgroups of the form $S_2 \wr S_m$. Since the nontrivial elements of $H$ are fixed-point-free the intersection between $H$ and a point stabilizer is trivial. Intersecting with $A_n$ we obtain $\Mindim(A_n) \leq 3$.

In the alternating group $A_7$ the intersection of the setwise stabilizers of $\{1,2,3,4\}$, $\{1,2,5\}$ and $\{2,5,6\}$ is trivial, and such stabilizers are maximal subgroups of $A_7$, therefore $\Mindim(A_7) \leq 3$. The first, second and fourth subgroups of the sixth conjugacy class of maximal subgroups of $A_{11}$ stored in \cite{gap} are isomorphic to $\mbox{M}_{11}$ and have trivial intersection, implying  $\Mindim(A_{11}) \leq 3$. Let now $n \in \{6,7,8,11,12\}$. In order to prove that $\Mindim(A_n) \geq 3$, using \cite{gap} we defined a function that builds the list of pairs $(A,B)$ where $A$ is a fixed representative of every conjugacy class of maximal subgroups of $A_n$ and $B$ is any maximal subgroup of $A_n$ distinct from $A$, and it returns true if and only if for each such pair there exists a maximal subgroup $C$ of $A_n$ belonging to the first conjugacy class of maximal subgroups of $A_n$ and with the property that $\{A,B,C\}$ is irredundant.
\end{proof}

We will use \cite[Corollary 1.5]{basesym}, which we state here for convenience.

\begin{teor} \label{balt}
Let $G=A_n$ with $n \geq 5$ acting primitively on a set with point stabilizer $H$. Then either $b(G)=2$ or $(n,H)$ is one of the following:
\begin{enumerate}
\item $H=(S_k \times S_{n-k}) \cap G$ with $k < n/2$.
\item $H=(S_2 \wr S_l) \cap G$ and $n=2l$.
\item $H=(S_k \wr S_l) \cap G$, $n=kl$ with $k \geq 3$, and either $l < k+2$ or $l=k+2 \in \{5,6\}$.
\item $(n,H)=(12,\mbox{M}_{12})$, $(11,\mbox{M}_{11})$, $(9,\PGammaL_2(8))$, $(8,\AGL_3(2))$, $(7,PSL_2(7))$, $(6,PSL_2(5))$, $(5,D_{10})$.
\end{enumerate}
In particular if $H$ acts primitively on $\{1,\ldots,n\}$ then $b(G)=2$ with finitely many exceptions.
\end{teor}

We proceed to the proof of Theorem 1. Observe that $A_4$ has a base of size $2$, so $\Mindim(A_4) = 2$. Now assume $n \geq 5$. By Lemma \ref{small} we know that $\Mindim(A_n)=3$ for $n \in \{6,7,8,11,12\}$. Also, $\Mindim(A_5)=2$ because the intransitive subgroups stabilizing $\{1,2\}$ and $\{1,3\}$ have trivial intersection, $\Mindim(A_9)=2$ by Theorem \ref{balt} because $A_9$ contains the primitive maximal subgroup $\ASL(2,3)$, which is different from $\PGammaL_2(8)$, and $\Mindim(A_{10})=2$ by Theorem \ref{balt} because $A_{10}$ contains a primitive maximal subgroup by Lemma \ref{arit2p}. This means that we may assume $n \geq 13$.

\ 

Observe that, being $n \geq 13$, by Theorem \ref{balt} we may assume that for all factorizations $n = k \ell$ with $k \geq 3$ we have $\ell \leq k+1$ or $\ell = k+2 \in \{5,6\}$. Also, if $n=p$ is a prime then letting $H$ to be any maximal subgroup of $G$ containing a $p$-cycle, $H$ acts primitively on $\{1,\ldots,p\}$ so $\Mindim(G)=2$. If $\ell = k+2 \in \{5,6\}$ then $n \in \{15,24\}$, and $A_{15}$, $A_{24}$ have proper primitive subgroups (with respect to the natural action on $\{1,\ldots,n\}$), so $\Mindim(A_{15}) = \Mindim(A_{24}) = 2$. Now assume that for all such factorizations we have $\ell \leq k+1$. Suppose $n$ is a product $abc$ with $2 \leq a \leq b \leq c$. If $b \geq 3$ then $ac \leq b+1 \leq c+1$, a contradiction. If $a=b=2$ then $n=4c$ so $c \leq 4+1=5$ and $n \in \{16,20\}$. However $A_{16}$ has a proper primitive subgroup $\AGL_2(4)$, and $A_{20}$ has a proper primitive subgroup $\PSL_2(19)$. We are left to discuss the case in which $n$ is of the form $pq$ with $p \geq q$ primes. If $q \neq 2$ then $q \leq p \leq q+1$ hence $p=q$ and $n=p^2$, $A_n$ has a maximal subgroup containing $\ASL_2(p)$, which acts primitively, so the corresponding primitive action of $A_n$ has a base of size $2$. Assume now that $n=2p$ with $p$ an odd prime, so that $p \geq 7$ being $n \geq 13$. If $A_n$ has primitive maximal subgroups then $\Mindim(A_n)=2$, now assume this is not the case. By Lemma \ref{arit2p}, $n$ is of the form specified in the statement.

\ 

Assume $n$ has the form $2p$ with $p \geq 7$ an odd prime such that $2p$ is not of the form $q+1$ with $q$ a prime power and $p \neq 11$. We claim that in this case $\Mindim(A_n)=3$. By Lemma \ref{small} it is enough to prove that $\Mindim(A_n) \geq 3$. By Lemma \ref{arit2p} the maximal subgroups of $S_n$ distinct from $A_n$ and the maximal subgroups of $A_n$ act non-primitively on $\{1,\ldots,n\}$. Moreover the maximal subgroups of $A_n$ are precisely the intersections between $A_n$ and the maximal subgroups of $S_n$. We will show that if $A$ and $B$ are maximal subgroups of $S_n$ then there exists a maximal subgroup $C$ of $S_n$ such that $\{A \cap A_n,B \cap A_n,C \cap A_n\}$ is irredundant.
\begin{enumerate}
\item Suppose $A$ and $B$ are of the form $S_p \wr S_2$. Say $A$ preserves a base block $\mathscr{A}$ and $B$ preserves a base block $\mathscr{B}$ (both of size $p$). By replacing, if needed, $\mathscr{A}$ with its complement we may assume $|\mathscr{A} \cap \mathscr{B}| \geq (p+1)/2 \geq 4$ so that $|\mathscr{A} \cup \mathscr{B}| \leq 2p-4$. Let $a \in \mathscr{A}-\mathscr{B}$, $i,j \in \mathscr{A} \cap \mathscr{B}$, $b \in \mathscr{B}-\mathscr{A}$, $\ell,t,s \in \{1,\ldots,2p\}-(\mathscr{A} \cup \mathscr{B})$. Let $C$ be the maximal intransitive subgroup stabilizing $\{b,i,\ell\}$ and its complement.
\begin{itemize}
\item $(ij)(\ell t) \in (A \cap B \cap A_n)-C$.
\item $(aj)(ts) \in (A \cap C \cap A_n)-B$.
\item $(bi)(ts) \in (B \cap C \cap A_n)-A$.
\end{itemize}

\item Suppose $A$ is of the form $S_p \wr S_2$ and $B$ is of the form $S_k \times S_{n-k}$ where $1 \leq k < p$. Say $A$ stabilizes a base block $\mathscr{A}$ of size $p$ and $B$ stabilizes $\mathscr{B}$ and its complement, with $|\mathscr{B}|$ larger than $p$.

Assume first that $\mathscr{A} \subseteq \mathscr{B}$. Pick $a,x,r,s \in \mathscr{A}$, $b \in \mathscr{B}-\mathscr{A}$ and $c \in \{1,\ldots,2p\}-\mathscr{B}$, and let $C$ be the intransitive maximal subgroup stabilizing $\{a,b,c\}$ and its complement.
\begin{itemize}
\item $(ax)(rs) \in (A \cap B \cap A_n)-C$.
\item $(bc)(rs) \in (A \cap C \cap A_n)-B$.
\item $(ab)(rs) \in (B \cap C \cap A_n)-A$.
\end{itemize}
Now assume $\mathscr{A} \not \subseteq \mathscr{B}$, and let $a \in \mathscr{A}-\mathscr{B}$. Since $|\mathscr{B}| > p$ there exist two distinct elements $b,x$ in $\mathscr{B}-\mathscr{A}$. Since $|\mathscr{B}|>p$, up to replacing, if needed, $\mathscr{A}$ with its complement we may assume that $|\mathscr{A} \cap \mathscr{B}| \geq (p+1)/2$, so that there exist distinct elements $i,j,r \in \mathscr{A} \cap \mathscr{B}$. Let $C$ be the maximal intransitive subgroup preserving $\{a,b,i\}$ and its complement.
\begin{itemize}
\item $(ir)(bx) \in (A \cap B \cap A_n)-C$.
\item $(ai)(jr) \in (A \cap C \cap A_n)-B$.
\item $(bi)(jr) \in (B \cap C \cap A_n)-A$.
\end{itemize}

\item Suppose $A$ is of the form $S_p \wr S_2$ and $B$ is of the form $S_2 \wr S_p$. Let $\mathscr{A} = \{1,\ldots,p\}$ be a base block for $A$, and let $\mathscr{B}_1,\ldots,\mathscr{B}_p$ be the base blocks corresponding to $B$, each of size $2$. Choose them so that $|\mathscr{B}_1 \cap \mathscr{A}| = 1$.

Suppose $|\mathscr{B}_i \cap \mathscr{A}|=1$ for all $i=1,\ldots,p$, without loss of generality $\mathscr{B}_i = \{i,p+i\}$. Let $C$ be the maximal intransitive subgroup stabilizing $\{1,2,p+1\}$ and its complement.
\begin{itemize}
\item $(1 \ldots p)(p+1 \ldots 2p) \in (A \cap B \cap A_n)-C$.
\item $(1\ p+1) (3\ p+3) \in (B \cap C \cap A_n)-A$.
\item $(12)(34) \in (A \cap C \cap A_n)-B$.
\end{itemize}
We may now assume that $\mathscr{B}_2 = \{r,s\} \subseteq \mathscr{A}$ and $\mathscr{B}_3 = \{y,z\}$ is disjoint from $\mathscr{A}$. Let $x \in \mathscr{B}_1-\mathscr{A}$, $t \in \mathscr{B}_1 \cap \mathscr{A}$, and let $C$ be the maximal intransitive subgroup stabilizing $\{x,y,t\}$ and its complement.
\begin{itemize}
\item $(yz)(rs) \in (A \cap B \cap A_n)-C$.
\item $(xy)(rs) \in (A \cap C \cap A_n)-B$.
\item $(xt)(rs) \in (B \cap C \cap A_n)-A$.
\end{itemize}
\item Suppose $A$ and $B$ are of the form $S_2 \wr S_p$. Since $A \neq B$ there are a base block $\mathscr{A}$ of $A$ and a base block $\mathscr{B}$ of $B$ distinct and with a non-empty intersection, without loss of generality $\mathscr{A}=\{1,2\}$ and $\mathscr{B}=\{1,3\}$. Let $\{3,x\}$ be the base block of $A$ containing $3$ and let $\{2,y\}$ be the base block of $B$ containing $2$. Let $C$ be the stabilizer of the point $1$.

First we prove that $A \cap C \cap A_n \not \subseteq B$ and $B \cap C \cap A_n \not \subseteq A$. If $x \neq y$ let $\{y,z\}$, $\{x,w\}$ be blocks of $A$ and $B$ respectively, we have $(3x)(yz) \in (A \cap C)-B$, $(2y)(xw) \in (B \cap C)-A$ and they also belong to $A_n$, if $x=y=4$ let $\{\alpha,\beta\}$, $\{\gamma,\delta\}$ be new blocks of $A$ and $B$ respectively, then $(34)(\alpha \beta) \in (A \cap C)-B$, $(24)(\gamma \delta) \in (B \cap C)-A$.

We are left to show that $A \cap B \cap A_n$ is not contained in $C$. Construct $\gamma \in A_n$ in the following way. Set $a_1=b_1=1$, $a_2=2$, $b_2=3$ and write the two block systems (corresponding to $A$ and to $B$) as follows: $$\{a_1,a_2\},\ \{b_2,a_3\},\ \{b_3,a_4\},\ \{b_4,a_5\}, \ldots,\ \{b_k,a_{k+1}\}, \ldots$$ $$\{a_1,b_2\},\ \{a_2,b_3\},\ \{a_3,b_4\},\ \{a_4,b_5\},\ \ldots, \{a_k,b_{k+1}\}, \ldots$$where $a_{k+1}=b_{k+1}$.

We want to construct $\gamma \in A_n$ fixing all the elements not belonging to $\{a_1,\ldots,a_{k+1},b_2,\ldots,b_k\}$ and with the property that $\gamma \in A \cap B$ and $\gamma(a_1)=a_3$, so that $\gamma(1) \neq 1$, implying $\gamma \not \in C$. As it turns out there is only one way to do so. If $k=2$ let $\gamma=(a_1 a_3)(a_2 b_2)$, if $k=3$ let $\gamma=(a_1 a_3 b_3)(a_2 b_2 a_4)$. If $k>3$ is odd let $$\gamma := (a_1 a_3 a_5 \ldots a_k b_k b_{k-2} b_{k-4} \ldots b_3) (a_2 b_2 b_4 b_6 \ldots b_{k-1} a_{k+1} a_{k-1} a_{k-3} \ldots a_4),$$whereas if $k>3$ is even let $$\gamma := (a_1 a_3 a_5 \ldots a_{k+1} b_{k-1} b_{k-3} \ldots b_3) (a_2 b_2 b_4 b_6 \ldots b_k a_k a_{k-2} a_{k-4} \ldots a_4).$$ Clearly $\gamma \in A_n$, since it is a product of two $k$-cycles in both cases. It is easy to show that $\gamma \in (A \cap B \cap A_n)-C$.
\item Suppose $A$ is of the form $S_2 \wr S_p$ and $B$ is of the form $S_k \times S_{n-k}$. Say $B$ stabilizes $\mathscr{B}$ and its complement with $|\mathscr{B}| > p$. The cardinality condition implies that $\mathscr{B}$ contains a base block $\mathscr{A}_1=\{x,y\}$. Suppose $\mathscr{B}$ intersects two base blocks $\mathscr{A}_2,\mathscr{A}_3$ in exactly one element, say $\mathscr{A}_2 = \{a,b\}$, $\mathscr{A}_3=\{z,w\}$ with $a,z \in \mathscr{B}$, $b,w \not \in \mathscr{B}$. Let $C$ be the maximal intransitive subgroup stabilizing $\{y,z,w\}$ and its complement.
\begin{itemize}
\item $(az)(bw) \in (A \cap B \cap A_n)-C$.
\item $(ab)(wz) \in (A \cap C \cap A_n)-B$.
\item $(ax)(yz) \in (B \cap C \cap A_n)-A$.
\end{itemize}
Suppose $\mathscr{B}$ contains two base blocks $\mathscr{A}_1=\{x,y\}$, $\mathscr{A}_2=\{a,b\}$ and intersects another base block $\mathscr{A}_3=\{z,w\}$ in exactly one element, say $z \in \mathscr{B}$, $w \not \in \mathscr{B}$. Let $C$ be the maximal intransitive subgroup stabilizing $\{y,z,w\}$ and its complement.
\begin{itemize}
\item $(ax)(by) \in (A \cap B \cap A_n)-C$.
\item $(ab)(wz) \in (A \cap C \cap A_n)-B$.
\item $(abx) \in (B \cap C \cap A_n)-A$.
\end{itemize}
We are left to discuss the case in which $\mathscr{B}$ is a union of some $\mathscr{A}_i$'s (blocks of $A$). In this case let $\{x,y\}$, $\{a,b\}$ be two blocks contained in $\mathscr{B}$ and let $\{t,w\}$ be a block disjoint from $\mathscr{B}$. Let $C$ be the maximal intransitive subgroup stabilizing $\{x,t\}$ and its complement.
\begin{itemize}
\item $(xy)(tw) \in (A \cap B \cap A_n)-C$.
\item $(yw)(xt) \in (A \cap C \cap A_n)-B$.
\item $(yab) \in (B \cap C \cap A_n)-A$.
\end{itemize}
\item Suppose $A$ and $B$ are maximal intransitive. Say $A$ stabilizes $\mathscr{A}$ and its complement, and $B$ stabilizes $\mathscr{B}$ and its complement.

Suppose first that $\mathscr{A} \subseteq \mathscr{B}$ or $\mathscr{B} \subseteq \mathscr{A}$. Without loss of generality assume the former case occurs. Up to exchanging $\mathscr{A}$ and $\mathscr{B}$ we may assume that $\mathscr{A}$ contains at least three distinct elements $a,x,y$. Let $b \in \mathscr{B}-\mathscr{A}$ and $z \in \{1,\ldots,n\}-\mathscr{B}$. Let $C$ be the maximal intransitive subgroup stabilizing $\{b,y,z\}$ and its complement.
\begin{itemize}
\item $(axy) \in (A \cap B \cap A_n)-C$.
\item $(ax)(bz) \in (A \cap C \cap A_n)-B$.
\item $(ax)(by) \in (B \cap C \cap A_n)-A$.
\end{itemize}
We may therefore assume that $\mathscr{A} \not \subseteq \mathscr{B}$ and $\mathscr{B} \not \subseteq \mathscr{A}$. Let $a \in \mathscr{A}-\mathscr{B}$, $b \in \mathscr{B}-\mathscr{A}$. By replacing, if needed, $\mathscr{A}$ or $\mathscr{B}$ or both with their complements we may assume that $|\mathscr{A} \cap \mathscr{B}| \geq (p+1)/2 \geq 4$. Let $x,y,z \in \mathscr{A} \cap \mathscr{B}$. Let $C$ be the maximal intransitive subgroup stabilizing $\{a,b,x\}$ and its complement.
\begin{itemize}
\item $(xyz) \in (A \cap B \cap A_n)-C$.
\item $(ax)(yz) \in (A \cap C \cap A_n)-B$.
\item $(xb)(yz) \in (B \cap C \cap A_n)-A$.
\end{itemize}
\end{enumerate}

\section{Acknowledgements}

We would like to thank the referee for carefully reading a previous version of this paper and having very useful suggestions.

\end{document}